\documentclass[11pt]{amsart}
 
\usepackage{hyperref}
\usepackage{amssymb,amsmath,color}

\theoremstyle{plain}
\newtheorem{thm}{Theorem}[section]

\newtheorem{lem}[thm]{Lemma}

\newtheorem{cor}[thm]{Corollary}
\newtheorem{prop}[thm]{Proposition}
\newtheorem{ass}{Assumption}

\theoremstyle{definition}

\theoremstyle{remark}
\newtheorem{rem}[thm]{Remark}

\newtheorem*{remark*}{Remark}

\numberwithin{equation}{section}

\newcommand{\cD}{{\mathcal{D}}}

\newcommand{\cF}{{\mathcal{F}}}

\newcommand{\cH}{{\mathcal{H}}}

\newcommand{\cL}{{\mathcal{L}}}

\newcommand{\cP}{{\mathcal{P}}}

\newcommand{\cS}{{\mathcal{S}}}
\newcommand{\cT}{{\mathcal{T}}}

        \newcommand{\field}[1]{{\mathbb{#1}}}
        \newcommand{\NN}{\field{N}}

        \newcommand{\RR}{\field{R}}
        \newcommand{\CC}{\field{C}}

\newcommand{\tr}{\mbox{\rm tr}}

\allowdisplaybreaks

\begin{document}

\title[Eigenvalue asymptotics for the Bochner Laplacian]{Semiclassical eigenvalue  asymptotics  for the Bochner Laplacian of a positive line bundle on a symplectic manifold}

\author[Y. A. Kordyukov]{Yuri A. Kordyukov}
\address{Institute of Mathematics\\
         Russian Academy of Sciences\\
         112~Chernyshevsky str.\\ 450008 Ufa\\ Russia} \email{yurikor@matem.anrb.ru}

\thanks{Supported by the Russian Science Foundation,
project no. 17-11-01004.}



\begin{abstract}
We consider the Bochner Laplacian on high tensor powers of a positive line bundle on a closed symplectic manifold (or, equivalently, the semiclassical magnetic Schr\"odinger operator with the non-degenerate magnetic field). We assume that the operator has discrete wells. The main result of the paper states asymptotic expansions for its low-lying eigenvalues. 
\end{abstract}

\date{}

 \maketitle

\tableofcontents

\section{Introduction}
Let $(X,g)$ be a compact Riemannian manifold of dimension $d$ without boundary and let $(L,h^L)$ be a Hermitian line bundle on $X$ with a Hermitian connection $\nabla^L$. For any $p\in \NN$, let $L^p:=L^{\otimes p}$ be the $p$th tensor power of $L$ and let
\[
\nabla^{L^p}: {C}^\infty(X,L^p)\to
{C}^\infty(X, T^*X \otimes L^p)
\] 
be the Hermitian connection on $L^p$ induced by $\nabla^{L}$. Consider the induced Bochner Laplacian $\Delta^{L^p}$ acting on $C^\infty(X,L^p)$ by
\[
\Delta^{L^p}=\big(\nabla^{L^p}\big)^{\!*}\,
\nabla^{L^p},
\]
where $\big(\nabla^{L^p}\big)^{\!*}: {C}^\infty(X,T^*X\otimes L^p)\to
{C}^\infty(X,L^p)$ is the formal adjoint of  $\nabla^{L^p}$.  If $\{e_j\}_{j=1,\ldots,d}$ is a local orthonormal frame of $TX$, then $\Delta^{L^p}$ is given by
\[
\Delta^{L^p}=-\sum_{j=1}^{d}\left[(\nabla^{L^p}_{e_j})^2-\nabla^{L^p}_{\nabla^{TX}_{e_j}e_j} \right],
\]
where $\nabla^{TX}$ is the Levi-Civita connection of the metric $g$. 

We are interested in the asymptotic behavior of the low-lying eigenvalues of the operator $\Delta^{L^p}$ as $p\to\infty$ (in the semiclassical limit). This problem is closely related with a similar problem for the semiclassical magnetic Schr\"odinger operator and was studied by many authors. For discussions about the subject, the reader is referred to the books and reviews \cite{FH10,HK-luminy,HK14,Raymond:book} and recent papers \cite{toeplitz,ma-savale,morin} (and the bibliography therein).

In this paper, we study the problem in a particular setting. Consider the real-valued closed 2-form $\omega$ given by 
\begin{equation}\label{e:def-omega}
\omega=\frac{i}{2\pi}R^L, 
\end{equation} 
where $R^L$ is the curvature of the connection $\nabla^L $ defined as $R^L=(\nabla^L)^2$. We assume that the 2-form $\omega$ is non-degenerate. Thus, $X$ is a symplectic manifold. In particular, its dimension is even, $d=2n$, $n\in \NN$. 

 Let $J_0 : TX\to TX$ be a skew-adjoint operator such that 
\[
\omega(u,v)=g(J_0u,v), \quad u,v\in TX. 
\]
Consider the function $\tau$ on $X$ given by
 \[
 \tau(x)=\pi \operatorname{Tr}[(-J^2_0(x))^{1/2}],\quad x\in X.
 \]
In general, $\tau$ is only continuous, but, if $\omega$ is of constant rank, then it is smooth. 
Put 
\[
\tau_0=\inf_{x\in X}\tau(x)>0.
\]
Our main assumption is the following assumption of discrete wells:

\begin{ass}\label{a:2}
Each minimum $x_0\in X$ of $\tau$, $\tau(x_0)=\tau_0$, is non-degenerate: 
\[
{\rm Hess}\,\tau(x_0)>0.
\]
\end{ass}

With each non-degenerate minimum $x_0$ of $\tau$, one can associate the model operator $\mathcal D_{x_0}$ in $L^2(T_{x_0}X)$ as follows.

Consider a second order differential operator $\mathcal L_{x_0}$ in $C^\infty(T_{x_0}X)$ given by
\begin{equation}\label{e:defL0}
\mathcal L_{x_0}=-\sum_{j=1}^{2n} \left(\nabla_{e_j}+\frac 12 R^L_{x_0}(Z,e_j)\right)^2-\tau (x_0), 
\end{equation}
where $\{e_j\}_{j=1,\ldots,2n}$ is an orthonormal base in $T_{x_0}X$ with the corresponding  coordinates on $T_{x_0}X\cong \mathbb R^{2n}$, $Z\in \mathbb R^{2n} \mapsto \sum_{j=1}^{2n}Z_je_j\in T_{x_0}X$. Here, for $U\in T_{x_0}X$, we denote by $\nabla_U$ the ordinary operator of differentiation in the direction $U$ on $C^\infty(T_{x_0}X)$. The operator $\mathcal L_{x_0}$ is well-defined, that is, it is independent of the choice of $\{e_j\}_{j=1,\ldots,2n}$. 

Let $\mathcal P=\mathcal P_{x_0}$ be the orthogonal projection in $L^2(T_{x_0}X)$ to the kernel $N=\ker \mathcal L_{x_0}$ of the operator $\mathcal L_{x_0}$. Its smooth Schwartz kernel $\mathcal P_{x_0}(Z, Z^\prime)$ with respect to the Riemannian volume form $dv_{TX}$ on $T_{x_0}X$ is called the Bergman kernel of $\mathcal L_{x_0}$. It can be described explicitly (see, for instance, \cite[Section 1.4]{ma-ma08} for more details). 

Consider the operator $J : TX\to TX$ given by 
\begin{equation}\label{e:defJ}
J=J_0(-J^2_0)^{-1/2}.
\end{equation} 
Then $J$ is an almost complex structure compatible with $\omega$ and $g$, that is, $g(Ju,Jv)=g(u,v), \omega(Ju,Jv)=\omega(u,v)$ and $\omega(u,Ju)\geq 0$ for any $u,v\in TX$.  
The almost complex structure $J_{x_0}: T_{x_0}X\to T_{x_0}X$ induces a splitting $T_{x_0}X\otimes_{\mathbb R}\mathbb C=T^{(1,0)}_{x_0}X\oplus T^{(0,1)}_{x_0}X$, where $T^{(1,0)}_{x_0}X$ and $T^{(0,1)}_{x_0}X$ are the eigenspaces of $J_{x_0}$ corresponding to eigenvalues $i$ and $-i$ respectively. Denote by $\det_{\mathbb C}$ the determinant function of the complex space $T^{(1,0)}_{x_0}X$. Put
\[
\mathcal J_{x_0}=-2\pi i J_0.
\]
Then $\mathcal J_{x_0} : T^{(1,0)}_{x_0}X\to T^{(1,0)}_{x_0}X$ is positive, and $\mathcal J_{x_0} : T_{x_0}X\to T_{x_0}X$ is skew-adjoint. 
The Bergman kernel $\mathcal P_{x_0}(Z, Z^\prime)$ is given by
\begin{multline} \label{e:defP0}
\mathcal P_{x_0}(Z, Z^\prime)\\ =\frac{\det_{\mathbb C}\mathcal J_{x_0}}{(2\pi)^n}\exp\left(-\frac 14\langle (\mathcal J^2_{x_0})^{1/2}(Z-Z^\prime), (Z-Z^\prime)\rangle +\frac 12 \langle \mathcal J_{x_0} Z, Z^\prime \rangle \right).
\end{multline}

Choose an orthonormal basis $\{w_j : j=1,\ldots,n\}$ of $T^{(1,0)}_{x_0}X$, consisting of eigenvectors of $\mathcal J_{x_0}$:
\[
\mathcal J_{x_0}w_j=a_jw_j,  \quad j=1,\ldots,n,
\]
with some $a_j>0$.
Then $e_{2j-1}=\frac{1}{\sqrt{2}}(w_j+\bar w_j)$ and $e_{2j}=\frac{i}{\sqrt{2}}(w_j-\bar w_j)$, $j=1,\ldots,n$, form an orthonormal basis of $T_{x_0}X$. We use this basis to define the coordinates $Z$ on $T_{x_0}X\cong \mathbb R^{2n}$ as well as the complex coordinates $z$ on $\mathbb C^{n} \cong \mathbb R^{2n}$, $z_j=Z_{2j-1}+iZ_{2j}, j=1,\ldots,n$.  In this coordinates, we get
\begin{equation}\label{e:defP} 
\mathcal P_{x_0}(Z, Z^\prime)=\frac{1}{(2\pi)^n}\prod_{j=1}^na_j \exp\left(-\frac 14\sum_{j=1}^na_j(|z_j|^2+|z_j^\prime|^2- 2z_j\bar z_j^\prime) \right).
\end{equation}

Denote by $Q_{x_0}$ the second order term in the Taylor expansion of $\tau$ at $x_0$ (in normal coordinates near $x_0$):
\[
Q_{x_0}(Z)=\left(\frac12 {\rm Hess}\,\tau(x_0)Z,Z\right)=\sum_{|\alpha|=2}(\partial^\alpha\tau)_{x_0}\frac{Z^\alpha}{\alpha!}, \quad Z\in T_{x_0}X\cong \mathbb R^{2n}.
\]
It is a positive definite quadratic form on $T_{x_0}X$.  

We will also need a smooth function $J_{1,2}$ on $X$, which appears in the leading coefficient of the diagonal asymptotic expansion of the generalized Bergman kernel associated with $\Delta^{L^p}$. We refer the reader to Section~\ref{s:3.3} (in particular, \eqref{e:Fq2} and below) for details. 

The model operator $\mathcal D_{x_0}$ associated with a non-degenerate minimum $x_0$ of $\tau$ is the Toeplitz operator in $L^2(T_{x_0}X)$ defined by
\begin{equation}\label{e:model-D}
\mathcal D_{x_0}=\mathcal P_{x_0}(Q_{x_0}(Z)+J_{1,2}(x_0)) : \ker \mathcal L_{x_0}\to \ker \mathcal L_{x_0}.
\end{equation}
It is an unbounded self-adjoint operator with discrete spectrum.  

Under Assumption \ref{a:2}, the minimum set $W=\{x\in X : \tau(x)=\tau_0\}$ is a finite set:
\[
W=\{x_1,\ldots, x_N\}.
\]
Let $\mathcal D$ be the self-adjoint operator on $L^2(T_{x_1}X)\oplus \ldots \oplus  L^2(T_{x_N}X)$ defined by 
\begin{equation}\label{e:def-Tau}
\mathcal D=\mathcal D_{x_1}\oplus \ldots \oplus \mathcal D_{x_N}.
\end{equation}
Denote by $\{\mu_j\}$ the increasing sequence of the eigenvalues of $\mathcal D$ (counted with multiplicities) and by $\{\lambda_j(\Delta^{L^p})\}$ the increasing sequence of the eigenvalues of the operator $\Delta^{L^p}$ (counted with multiplicities).

\begin{thm}\label{t:Bochner-expansion}
For any $j\in \NN$, we have
\begin{equation}\label{e:l1}
\lambda_j(\Delta^{L^p})= p\tau_0+\mu_j+\mathcal O(p^{-1/2}), \quad p\to\infty.
\end{equation}
If $\mu_j$ is a simple eigenvalue of $\mathcal D$, then $\lambda_j(\Delta^{L^p})$ admits a complete asymptotic expansion
\begin{equation}\label{e:l1s}
\lambda_j(\Delta^{L^p})= p\tau_0+\mu_j+\sum_{k=1}^\infty a_{k,j}p^{-k/2}, \quad p\to\infty.
\end{equation}
\end{thm}

In the case when the line bundle is trivial and the minimum is unique, a similar result is obtained in the recent work \cite{morin} by a different method. In dimension two, these results were obtained in \cite{HM01,HK-Shubin} (see also \cite{HK15,RV15} for the case of excited states).

\begin{rem}
Recall the relation of our Bochner Laplacian with the semiclasssical magnetic Schr\"odinger operator. Assume that the Hermitian line bundle $(L,h^L)$ on $X$ is trivial, that is, $L=X\times \CC$ and $|(x,z)|^2_{h^L}=|z|^2$ for $(x,z)\in X\times \CC$. Then the Hermitian connection $\nabla^L$ can be written as $\nabla^L=d-i \mathbf A$ with some real-valued 1-form $\mathbf A$ (the magnetic potential). Its curvature $R^L$ is given by 
\begin{equation}\label{e:RLB}
R^L=-i\mathbf B,
\end{equation}
where $\mathbf B=d\mathbf A$ is a real-valued 2-form (the magnetic field). For the form $\omega$ defined by \eqref{e:def-omega}, we have
\[
\omega=\frac{1}{2\pi}\mathbf B. 
\]
The associated Bochner Laplacian $\Delta^{L^p}$ is related with the semiclassical magnetic Schr\"odinger operator
\[
H^{\hbar}=(i\hbar d+\mathbf A)^*(i\hbar d+\mathbf A), \quad \hbar>0
\]
by the formula
\[
\Delta^{L^p}=\hbar^2H^{\hbar}, \quad \hbar=\frac{1}{p},\quad p\in \NN. 
\]
Let $B : TX\to TX$ be a skew-adjoint endomorphism such that 
\[
\mathbf B(u,v)=g(Bu,v), \quad u,v\in TX. 
\]
Then we have 
\[
J_0=\frac{1}{2\pi}B,\quad J=B(B^*B)^{-1/2}, \quad
\mathcal J=-iB.
\]
Finally, the function $\tau$ coincides with the magnetic intensity:
 \[
 \tau=\frac 12 \operatorname{Tr}(B^*B)^{1/2}= \operatorname{Tr}^+(B). 
 \]

The assumption that the 2-form $\omega$ is non-degenerate means that the magnetic field is non-degenerate. Assumption~\ref{a:2} states that the magnetic field has discrete wells. 
 \end{rem}

\begin{rem} 
Recall the description of the Toeplitz operator $\cT(Q)=\cP Q : \ker \mathcal L_{x_0}\to \ker \mathcal L_{x_0}$ in $L^2(\RR^{2n})$, where $Q=Q(z,\bar z)$ is a polynomial in $\CC^n\cong \RR^{2n}$, the operator $\mathcal L_{x_0}$ is given by \eqref{e:defL0} and $\cP$ is given by \eqref{e:defP}. We refer to \cite{toeplitz} for more details.  

First, using a simple scaling, this operator is unitarily equivalent to the Toeplitz operator $\mathcal T^0(P)=\Pi P$ in the Fock space $\cF_n\subset L^2(\CC^n; e^{-\frac{1}{2}|z|^2}dz)$ with $P$ given by
\[
P(z,\bar z)=Q\left(\frac{\sqrt{2}}{\sqrt{a_1}}z_1, \ldots, \frac{\sqrt{2}}{\sqrt{a_n}}z_n, \frac{\sqrt{2}}{\sqrt{a_1}}\bar z_1, \ldots, \frac{\sqrt{2}}{\sqrt{a_n}}\bar z_n\right), \quad z\in \CC^n.
\] 

Next, we use the well-known relation between the Bargmann-Fock and the Schr\"odinger presentations via the Bargmann transform $B: L^2(\RR^n)\to \cF_n$. In the terminology of \cite{Berezin71}, for a polynomial $P$ in $\CC^n$, the operator $B^{-1}\cT^0(P)B$ is the differential operator with polynomial coefficients in $\RR^n$ with anti-Wick symbol $P(\bar z,z)$. One can compute the Weyl symbol of this operator by the well-known formula (see, for instance, \cite{Berezin71}). In particular, if $P$ is a positive definite quadratic form, then
\[
B^{-1}\cT^0(P)B=Op_w(\tilde{P})+\frac{\operatorname{tr}(\tilde{P})}{2}, 
\]
where $\tilde P$ is a quadratic form on $\RR^{2n}\cong T^*\RR^n$, corresponding to $P$ under the linear isomorphism $(x,\xi)\in \RR^{2n}\mapsto z\in \CC^n$, $z_k=\frac{1}{\sqrt{2}}(x_k-i\xi_k), k=1,\ldots,n$, $Op_w(\tilde{P})$ is the pseudodifferential operator in $\RR^n$ with Weyl symbol $\tilde{P}$.

In the case when $n=1$ and $Q$ is a positive definite quadratic form on $\CC\cong \RR^2$, the eigenvalues of $\mathcal T(Q)$ are given by:
%\begin{equation}\label{e:lambda-T}
\[
\lambda_j=\frac{2\sqrt{D}}{\tau_0}j+\frac{A^2}{2\tau_0}, \quad j=0,1,\ldots,
\]
%\end{equation}
where $D=\det Q$, $A=\operatorname{tr} Q^{1/2}$.
\end{rem}

The paper is organized as follows. In Section~\ref{s:upper} we construct approximate eigenfunctions of the Bochner Laplacian $\Delta^{L^p}$. This allows us to prove upper bounds for its low-lying eigenvalues. In Section~\ref{s:lower} we prove lower bounds for the  low-lying eigenvalues of $\Delta^{L^p}$ and complete the proof of Theorem~\ref{t:Bochner-expansion}. The proof of lower bounds is based on several techniques. First, we prove some Agmon type estimates and localization results for the associated eigensections. 
Then we consider the renormalized Bochner Laplacian $\Delta_p=\Delta^{L^p}-p\tau$ and the associated generalized Bergman projection $P_{\mathcal H_p}$ \cite{Gu-Uribe,ma-ma08}. Under weak assumptions, we show the asymptotic behavior of the low-lying eigenvalues of $\Delta^{L^p}$ is equivalent to the asymptotic behavior of the low-lying eigenvalues of $P_{\mathcal H_p}\Delta^{L^p}P_{\mathcal H_p}$. Recall that $P_{\mathcal H_p}$ is the orthogonal projection to the eigenspace of $\Delta_p$, corresponding to the eigenvalues localized near zero uniformly in $p$. Therefore, this result can be considered as the reduction of the problem to the lowest Landau level. Finally, we use the fact that the operator $P_{\mathcal H_p}\Delta^{L^p}P_{\mathcal H_p}$ belongs to the algebra of Toeplitz operators on $X$ associated with the renormalized Bochner Laplacian $\Delta_p$, which was constructed in \cite{ioos-lu-ma-ma,bergman}. We apply the results of \cite{toeplitz} on eigenvalue asymptotics of the low-lying eigenvalues of Toeplitz operators with discrete wells on symplectic manifolds. 

\section{Upper bounds} \label{s:upper}

\subsection{Approximate eigenfunctions: main result}
The purpose of this section is to prove the following accurate upper
bound for the eigenvalues of the operator $\Delta^{L^p}$.

\begin{thm}\label{t:qmodes}
Under Assumption \ref{a:2}, for any natural $j$, there exist $\phi^p_{j2}\in C^\infty(X)$,
$C_{j,2}>0$ and $p_{j,2}>0$ such that
\begin{equation}\label{e:orth2}
(\phi^p_{j_12},\phi^p_{j_22}) =\delta_{j_1j_2}+
\mathcal O_{j_1,j_2}(p^{-1})
\end{equation}
and, for any $p>p_{j,2}$,
\begin{equation}\label{e:Hh2}
\|\Delta^{L^p}\phi^p_{j2}- (p\tau_0+\mu_j)\phi^p_{j2}\|\leq
C_{j,2}p^{-\frac{1}{2}}\|\phi^p_{j2}\|,
\end{equation}
where $\mu_j$ is the $j$th eigenvalue of $\mathcal D$.

If, for some $j$, $\mu_j$ is a simple eigenvalue of $\mathcal D_{x_0}$ for some $x_0$,  then there exists a sequence $\{\mu_{j,\ell} : \ell=0,1,2,\ldots\}$ with
\[
\mu_{j,0} = \tau_0, \quad \mu_{j,1}=0,\quad \mu_{j,2}=\mu_j,
\]
and for any $N$, there exist $\phi^p_{jN}\in C^\infty(X)$,
$C_{j,N}>0$ and $p_{j,N}>0$ such that
\begin{equation}\label{e:orth}
(\phi^p_{j_1N},\phi^p_{j_2N}) =\delta_{j_1j_2}+
\mathcal O_{j_1,j_2}(p^{-1})
\end{equation}
and, for any $p>p_{j,N}$,
\begin{equation}\label{e:Hh}
\|\Delta^{L^p}\phi^p_{jN}- \mu_{jN}^p \phi^p_{jN}\|\leq
C_{j,N}p^{-\frac{N-1}{2}}\|\phi^p_{jkN}\|,\,
\end{equation}
where
\[
\mu_{jN}^p=p\sum_{\ell=0}^N \mu_{j,\ell} p^{-\frac \ell 2}
\]
\end{thm}

Since the operator $\Delta^{L^p}$ is self-adjoint, using Spectral Theorem, we
immediately deduce the existence of eigenvalues near the points
$\mu_{jN}^p$.

\begin{cor}\label{c:dist}
For any natural $j$, there exist $C_{j2}>0$ and
$p_{j2}>0$ such that, for any $p>p_{j2}$,
\[
{\rm dist}(p\tau_0+\mu_j, \sigma(\Delta^{L^p}))\leq
C_{j2}p^{-\frac{1}{2}}.
\]
If, for some $j$, $\mu_j$ is a simple eigenvalue of $\mathcal D$,  then, for any natural $N$, there exist $C_{jN}>0$ and $p_{jN}>0$ such that, for any $p>p_{jN}$,
\[
{\rm dist}(\mu_{jN}^p, \sigma(\Delta^{L^p}))\leq
C_{jN}p^{-\frac{N-1}{2}}.
\]
\end{cor}

\begin{rem}\label{r:upper-bound}
As an immediate consequence of Theorem~\ref{t:qmodes}, we
deduce that, for any natural $j$, there exist $C_{j2}>0$ and $p_{j2}>0$
such that, for any $p>p_{j2}$,  
\[
\lambda_j(\Delta^{L^p})\leq p\tau_0+\mu_j + C_{j2}p^{-\frac{1}{2}}.
\]
If, for some $j$, $\mu_j$ is a simple eigenvalue of $\mathcal D$,  then, for any natural $N$, there exist $C_{jN}>0$ and $p_{jN}>0$ such that, for any $p>p_{jN}$, 
\[
\lambda_j(\Delta^{L^p})\leq \mu_{jN}^p + C_{jN}p^{-\frac{N-1}{2}}.
\]

In particular, this implies the upper bound in Theorem~\ref{t:Bochner-expansion}.
\end{rem}

The rest of this section is devoted to the proof of Theorem \ref{t:qmodes}.

\subsection{Rescaling and formal asymptotic expansions}\label{s:rescaling}
Fix $j\in \NN$. Without loss of generality, we can assume that $\mu_j$ is an eigenvalue of $\mathcal D_{x_0}$ for some $x_0\in W$. The approximate eigensections $\phi^p_{jk}\in C^\infty(X.L^p)$, which we are going to construct, will be supported in a small neighborhood of $x_0$. So, in a neighborhood of $x_0$, we will consider some special local coordinate system with coordinates $Z\in \mathbb R^{2n}$ such that $x_0$ corresponds to $0$ and a trivialization oof the line bundle $L$ over it. We will only apply our operator on functions which are a product of cut-off functions with functions of the form of linear combinations of terms like $p^{-\nu}w(p^{1/2}Z)$ with $w$ in $\cS(\RR^{2n})$. These functions are consequently $O(p^{-\infty})$ outside a fixed neighborhood of $0$. We will start by doing the computations formally in the sense that everything is determined modulo $O(p^{-\infty})$, and any smooth function will be replaced by its Taylor's expansion. It is then easy to construct non formal approximate eigenfunctions. For the construction of the local coordinate system and computation of a formal asymptotic expansion of the Bochner Laplacian in rescaled normal coordinates, we will use results of \cite[Sections 1.1 and 1.2]{ma-ma08}, which are inspired by the analytic localization technique of Bismut-Lebeau \cite{BL}.  

First, we introduce normal coordinates near $x_0$. Let $a^X$ be the injectivity radius of $(X,g)$. We will identify $B^{T_{x_0}X}(0,a^X)$ with $B^{X}(x_0,a^X)$ by the exponential map $\operatorname{exp}^X : T_{x_0}X\to X$. For $Z\in B^{T_{x_0}X}(0,a^X)$ we identify $L_Z$ to $L_{x_0}$ by parallel transport with respect to the connection $\nabla^L$ along the curve $\gamma_Z : [0,1]\ni u \to \exp^X_{x_0}(uZ)$. Consider the line bundle $L_0$ with fibers $L_{x_0}$ on $T_{x_0}X$. Denote by $\nabla^{L_0}$, $h^{L_0}$ the connection and the metric on the restriction of $L_0$ to $B^{T_{x_0}X}(0,a^X)$ induced by the identification $B^{T_{x_0}X}(0,a^X)\cong B^{X}(x_0,a^X)$ and the trivialization of $L$ over $B^{T_{x_0}X}(0,a^X)$. 

Let $\{e_j\}_{j=1,\ldots,2n}$ be an oriented orthonormal basis of $T_{x_0}X$. It gives rise to an isomorphism $X_0:=\mathbb R^{2n}\cong T_{x_0}X$.  Denote by $dv_{TX}$ the Riemannian volume form of $(T_{x_0}X, g^{T_{x_0}X})$ and by $dv_X$ the volume form on $B^{T_{x_0}X}(0,a^X)$, corresponding to the Riemannian volume form on $B^{X}(x_0,a^X)$ under identification $B^{T_{x_0}X}(0,a^X)\cong B^{X}(x_0,a^X)$. Then we have 
\[
dv_{X}(Z)=\kappa_{x_0}(Z)dv_{TX}(Z), \quad Z\in B^{T_{x_0}X}(0,a^X),
\] 
with some smooth positive function $\kappa_{x_0}$ on $B^{T_{x_0}X}(0,a^X)$.

Now we use a rescaling introduced in \cite[Section 1.2]{ma-ma08}. 
Denote $t=\frac{1}{\sqrt{p}}$. For $s\in C^\infty(\mathbb R^{2n})$, set
\[
S_ts(Z)=s(Z/t), \quad Z\in \mathbb R^{2n}.
\]
The rescaled connection $\nabla_t$ is defined as 
\[
\nabla_t=tS^{-1}_t\kappa^{\frac 12}\nabla^{L_0^p}\kappa^{-\frac 12}S_t.
\]
Let $\Gamma^L$ be the connection form of $\nabla^L$ with respect to some fixed frames for $L$ which are parallel along the curves $\gamma_Z : [0,1]\ni u \to \exp^X_{x_0}(uZ)$ under the chosen trivializations on $B^{T_{x_0}X}(0,a^X)$. Then on $B^{T_{x_0}X}(0,a^X/t)$
\begin{align*}
\nabla_{t,e_i}&=\kappa^{\frac 12}(tZ)\left(\nabla_{e_i}+\frac{1}{t}\Gamma^L(e_i)(tZ)\right)\kappa^{-\frac 12}(tZ)\\
&=\nabla_{e_i}+\frac{1}{t}\Gamma^L(e_i)(tZ)-t\left(\kappa^{-1}(e_i\kappa)\right)(tZ).
\end{align*}
Recall that 
\[
\sum_{|\alpha|=r}(\partial^\alpha\Gamma^L)_{x_0}(e_j)\frac{Z^\alpha}{\alpha!}=\frac{1}{r+1}\sum_{|\alpha|=r-1}(\partial^\alpha  R^L)_{x_0}(\mathcal R,e_j)\frac{Z^\alpha}{\alpha!}.
\]
where $\mathcal R(Z)=\sum_{j=1}^{2n} Z_je_j\in \mathbb R^{2n}\cong T_ZX_0$ denote the radial vector field on $X_0$.  In particular, 
\[
\Gamma^L(e_j)(Z)=\frac{1}{2}R^L_{x_0}(\mathcal R,e_j)+O(|Z|^2). 
\]
For the rescaled operator $\mathcal H_t$ on $B^{T_{x_0}X}(0,a^X/t)$ defined as 
\begin{equation}\label{scaling}
\mathcal H_t=t^2S^{-1}_t\kappa^{\frac 12}\Delta^{L^p}\kappa^{-\frac 12}S_t,
\end{equation}
we get
\[
\mathcal H_t=-\sum_{j,k=1}^{2n} g^{jk}(tZ)\left[\nabla_{t,e_j}\nabla_{t,e_k}- t\sum_{\ell=1}^{2n}\Gamma^{\ell}_{jk}(tZ)\nabla_{t,e_\ell}\right].
\] 

Let us develop the coefficients of the rescaled operator $\mathcal H_t$ in Taylor series in $t$. The resulting asymptotic expansion is given in \cite[Theorem 1.4]{ma-ma08}. For any $m\in \NN$, we get
\begin{equation}
\mathcal H_t=\mathcal H^{(0)}+\sum_{j=1}^m \mathcal H^{(j)}t^j+\mathcal O(t^{m+1}), 
\end{equation}
where there exists $m^\prime\in \NN$ so that for every $k\in\NN$ and $t\in [0,1]$ the derivatives up to order $k$ of the coefficients of the operator $\mathcal O(t^{m+1})$ are bounded by $Ct^{m+1}(1+|Z|)^{m^\prime}$. 

The leading term $\mathcal H^{(0)}$ is given by
\[ 
\mathcal H^{(0)}=-\sum_{j=1}^{2n} \left(\nabla_{e_j}+\frac 12 R^L_{x_0}(Z,e_j)\right)^2=\mathcal L_{x_0}+\tau_0.
\]
The spectrum of $\mathcal H^{(0)}$ consists of discrete set of eigenvalues of infinite multliplicity (see, for instance, \cite[Theorem 1.15]{ma-ma08}). In particular, the lowest eigenvalue of $\mathcal H^{(0)}$ is $\tau_0$. 

For the next terms, we have
\[
\mathcal H^{(j)}=\mathcal O_j+\sum_{|\alpha|=j}(\partial^\alpha\tau)_{x_0}\frac{Z^\alpha}{\alpha!},\quad j=1,2,\ldots,
\]
where $\mathcal O_j$ are second order differential operators with polynomial coefficients. In particular, 
\[
\mathcal H^{(1)}=\mathcal O_1,\quad  \mathcal H^{(2)}=\mathcal O_2+Q_{x_0}(Z).
\]
Explicit formulas for $\mathcal O_1$ and $\mathcal O_2$ are given in \cite[Theorem 1.4]{ma-ma08}. We refer the reader to \cite{ma-ma08} for more details. 

\subsection{Construction of approximate eigenfunctions}
First, we construct a formal eigenfunction $u_t$ of the operator
$\mathcal H_t$ admitting an asymptotic expansion in the form of a
formal asymptotic series in powers of $t$
\[
u_t=\sum_{\ell=0}^\infty u_\ell t^{\ell}, \quad u^{(\ell)}\in
{\cS(\RR^{2n})},
\]
with the corresponding formal eigenvalue
\[
\lambda_t=\sum_{\ell=0}^\infty\lambda_\ell t^{\ell},
\]
such that
\[
\mathcal H_tu_t-\lambda_t u_t=0
\]
in the sense of asymptotic series in powers of $t$.
\medskip\par
{\bf The first terms.} Looking at the coefficient of $t^0$,  we
obtain:
\[
\mathcal H^{(0)}u_0=\lambda_0u_0\,.
\]
Thus, we take
\begin{equation}\label{e:0}
\lambda_0=\tau_0, \quad
u_0=\mathcal Pu_0\in N, \quad \|u_0\|=1,
\end{equation}
where $u_0$ will be determined later.

Looking at the coefficient of $t^{1}$,  we obtain:
\[
\mathcal H^{(0)}u_1+\mathcal H^{(1)}u_0=\lambda_0u_1+\lambda_1u_0 \Leftrightarrow \mathcal L_{x_0}u_1=\lambda_1u_0-\mathcal O_1u_0.
\]
The orthogonality condition implies that
\[
\lambda_1=\mathcal P\mathcal O_1u_0.
\]
By \cite[Thm 1.16]{ma-ma08}, we have $\mathcal P\mathcal O_1\mathcal P=0$. Therefore
\[
\lambda_1=0.
\]
Under this condition, we get
\begin{equation}\label{e:1/2}
u_1=v_1+v_1^\bot,
\end{equation}
where
\[
v_1^\bot=-\mathcal L_{x_0}^{-1}\mathcal O_1u_0
\]
and $v_1\in N$, $v_1\bot u_0$, will be determined later.

Next, the cancelation of the coefficient of $t^2$ gives:
\begin{equation}\label{e:1}
\mathcal L_{x_0}u_2= \lambda_2u_0-\mathcal O_1u_1-\mathcal H^{(2)}u_0.
\end{equation}
The orthogonality condition for \eqref{e:1} implies that
\begin{equation}\label{e:1ort}
\lambda_2u_0-\mathcal P \mathcal O_1u_1-\mathcal P\mathcal H^{(2)}u_0=0.
\end{equation}

Consider the operator $\mathcal D_{x_0} : N\to N$ given by
\begin{equation}\label{e:T}
\mathcal D_{x_0}=\mathcal P\mathcal H^{(2)}-\mathcal P\mathcal O_1 \mathcal L_{x_0}^{-1}\mathcal O_1. 
\end{equation}
Observe that 
\[
\mathcal D_{x_0}=F_{1,2}+\mathcal P Q_{x_0}(Z),
\]
where (cf. \cite[Eq. (1.109)]{ma-ma08})
\begin{equation}\label{e:F12}
F_{1,2}=\mathcal P\mathcal O_2-\mathcal P \mathcal O_1\mathcal L_{x_0}^{-1}\mathcal O_1. 
\end{equation}
One can show that the operator $F_{1,2}: N\to N$ is a scalar operator:
\[
F_{1,2}=J_{1,2}(x_0)\in \RR,
\]
with the function $J_{1,2}\in C^\infty(X)$ mentioned in Introduction (see Section~\ref{s:3.3}, in particular, \eqref{e:Fq2} and below). So this definition of $\mathcal D_{x_0}$ agrees with \eqref{e:model-D}.

By \eqref{e:1/2}, the equation \eqref{e:1ort} can be rewritten as
\[
\mathcal D_{x_0} u_0=\lambda_2u_0.
\]
Thus, we put
\[
\lambda_2=\mu_{j},
\]
where $\mu_{j}$ is the fixed eigenvalue of $\mathcal D_{x_0}$ and
\[
u_0=\Psi_{j},
\]
where $\Psi_{j}\in N$ is a normalized eigenfunction of $\mathcal D_{x_0}$ associated to $\mu_{j}$.

Moreover, we conclude that $u_2$ is a solution of \eqref{e:1},
which can be written as
\begin{equation}\label{e:u2}
u_2=v_2+v_2^\bot,
\end{equation}
where $v_2^\bot\in N^\bot$ is a solution of \eqref{e:1}:
\[
v_2^\bot=\mathcal L_{x_0}^{-1}(\lambda_2u_0-\mathcal P^\bot\mathcal O_1u_1-\mathcal P^\bot\mathcal H^{(2)}u_0),\quad \mathcal P^\bot=1-\mathcal P,
\]
and $v_2\in N$ will be determined later. By \eqref{e:1/2} and \eqref{e:1ort}, we have
\[
v_2^\bot=-\mathcal L_{x_0}^{-1}\mathcal P^\bot\mathcal O_1v_1+f_2,
\]
where $f_2$ is known:
\[
f_2=\mathcal L_{x_0}^{-1}(\mathcal P^\bot\mathcal O_1\mathcal L_0^{-1}\mathcal O_1-\mathcal P^\bot\mathcal H^{(2)})u_0.
\]

Now the cancelation of the coefficient of $t^{3}$ gives:
\begin{equation}\label{e:3/2}
\mathcal L_{x_0}u_3= \lambda_3u_0+\lambda_2u_1-\mathcal H^{(3)}u_0-\mathcal H^{(2)}u_1-\mathcal O_1u_2.
\end{equation}
The orthogonality condition for \eqref{e:3/2} reads as
\begin{equation}\label{e:3/2ort}
\lambda_3u_0+\lambda_2v_1-\mathcal P\mathcal H^{(3)}u_0-\mathcal P\mathcal H^{(2)}u_1-\mathcal P\mathcal O_1u_2=0.
\end{equation}
Using \eqref{e:1/2} and \eqref{e:u2}, it can be written as
\begin{multline}\label{e:3/2ort2}
\mathcal D_{x_0}v_1-\lambda_2v_1\\ =
\lambda_3u_0-\mathcal P\mathcal H^{(3)}u_0+\mathcal P\mathcal H^{(2)}\mathcal L_{x_0}^{-1}\mathcal O_1u_0-\mathcal P\mathcal O_1\mathcal L_{x_0}^{-1}(\mathcal P^\bot\mathcal O_1\mathcal L_{x_0}^{-1}\mathcal O_1-\mathcal P^\bot\mathcal H^{(2)})u_0.
\end{multline}

In general, this equation may have no solution for any choice of $\lambda_3$. So we have to stop here.  Put
\[
u_t^{(2)}= u_0 + u_1t+u_2t^2, \quad
\lambda_t^{(1)}=\lambda_0+\lambda_2t^2.
\]
Then we have
\[
\mathcal H_tu_t^{(2)}-\lambda_t^{(2)}u_t^{(2)}=O(t^{3}).
\]

Assume that $\lambda_2=\mu_j$ is a simple eigenvalue of $\mathcal D_{x_0}$. Then we can proceed further. The equation \eqref{e:3/2ort2} has a solution $v_1$ if and only if:
\[
\lambda_3=\langle (\mathcal P\mathcal H^{(3)}-\mathcal P\mathcal H^{(2)}\mathcal L_{x_0}^{-1}\mathcal O_1+\mathcal P\mathcal O_1\mathcal L_{x_0}^{-1}(\mathcal P^\bot\mathcal O_1\mathcal L_{x_0}^{-1}\mathcal O_1-\mathcal P^\bot\mathcal H^{(2)}))u_0, \Psi_j\rangle.
\]
Under this condition, there exists a unique solution $v_1$ of \eqref{e:3/2ort}, orthogonal to $\Psi_j$. 

With these choices, we find a solution $u_3$ of \eqref{e:3/2} of the form
\[
u_3=v_3+v_3^\bot,
\]
where $v_3^\bot\in N^\bot$ is given by
\[
v_3^\bot=\mathcal L_{x_0}^{-1}(\lambda_2v^\bot_1-\mathcal P^\bot\mathcal H^{(3)}u_0-\mathcal P^\bot \mathcal H^{(2)}u_1-\mathcal P^\bot\mathcal O_1u_2),
\]
and $v_3\in N$ will be determined later. Observe that $v_3^\bot$ has the form
\[
v_3^\bot=-\mathcal L_{x_0}^{-1}\mathcal P^\bot\mathcal O_1v_2+f_3,
\]
where $f_3$ is known:
\[
f_3=\mathcal L_{x_0}^{-1}(\lambda_2v^\bot_1-\mathcal P^\bot\mathcal H^{(3)}u_0-\mathcal P^\bot\mathcal H^{(2)}u_1-\mathcal P^\bot\mathcal O_1v^\bot_2).
\]
\medskip\par
{\bf The iteration procedure.} Suppose that the coefficients of
$t^{\ell}$ equal zero for $\ell=0,\ldots,k-1$, for some $k>3$. Then we know
the coefficients $\lambda_\ell$ for $\ell=0,\ldots,k-1$. We also
know that $u_\ell$ for $\ell=0,\ldots,k-1$ can be written as
\begin{equation}\label{e:ul}
u_\ell=v_\ell+v^\bot_\ell,
\end{equation}
where $v^\bot_\ell, \ell=0,\ldots,k-1$, are some functions
in $N^\bot$ of the form
\begin{equation}\label{e:vl-bot}
v_\ell^\bot=-\mathcal L_{x_0}^{-1}\mathcal P^\bot\mathcal O_1v_{\ell-1}+f_\ell,
\end{equation}
with some known $f_\ell\in N^\bot$, and $v_\ell \in N$ are known for $\ell=0,\ldots,k-3$, $v_\ell \perp \Psi_j$.

The cancelation of the coefficient of $t^{k}$ gives:
\begin{multline}\label{e:n/2}
\mathcal L_{x_0}u_k= \lambda_ku_0-\mathcal H^{(k)}u_0 +\sum_{\ell=3}^{k-1}\lambda_\ell u_{k-\ell}-\sum_{\ell=3}^{k-1} \mathcal H^{(\ell)}u_{k-\ell}\\ +\lambda_2u_{k-2}-\mathcal H^{(2)}u_{k-2}-\mathcal O_1u_{k-1}.
\end{multline}
The orthogonality condition for \eqref{e:n/2} reads as
\begin{multline}\label{e:n/2ort}
\lambda_ku_0-\mathcal P\mathcal H^{(k)}u_0 +\sum_{\ell=3}^{k-1}\lambda_\ell v_{k-\ell}-\sum_{\ell=3}^{k-1}\mathcal P \mathcal H^{(\ell)}u_{k-\ell}\\ +\lambda_2v_{k-2}-\mathcal P \mathcal H^{(2)}u_{k-2}-\mathcal P\mathcal O_1u_{k-1}=0.
\end{multline}
Using \eqref{e:ul} and \eqref{e:vl-bot}, \eqref{e:n/2ort} can be written as
\begin{multline}\label{e:lort2}
\mathcal D_{x_0}v_{k-2}-\lambda_2v_{k-2}=\lambda_ku_0-\mathcal P\mathcal H^{(k)}u_0+\sum_{\ell=3}^{k-1}\lambda_\ell v_{k-\ell}-\sum_{\ell=3}^{k-1}\mathcal P \mathcal H^{(\ell)}u_{k-\ell}\\ -\mathcal P \mathcal H^{(2)}v^\bot_{k-2}-\mathcal P\mathcal O_1f_{k-1}.
\end{multline}
The equation \eqref{e:lort2} has a solution $v_{k-2}$ if and only if:
\[
\lambda_k=-\left\langle \left(-\mathcal P\mathcal H^{(k)}u_0-\sum_{\ell=3}^{k-1}\mathcal P \mathcal H^{(\ell)}u_{k-\ell}-\mathcal P \mathcal H^{(2)}v^\bot_{k-2}-\mathcal P\mathcal O_1f_{k-1}\right), \Psi_j\right\rangle.
\]
Under this condition, there exists a unique solution $v_{k-2}$ of \eqref{e:lort2}, orthogonal to $\Psi_j$. 

With these choices, we find a solution $u_k$ of \eqref{e:n/2} of the form
\[
u_k=v_k+v_k^\bot,
\]
where $v_k^\bot\in N^\bot$ is given by
\[
v^\bot_k=\mathcal L_{x_0}^{-1}\left(-\mathcal P^\bot\mathcal H^{(k)}u_0+\sum_{\ell=2}^{k-1}\lambda_\ell v^\bot_{k-\ell}-\sum_{\ell=2}^{k-1} \mathcal P^\bot \mathcal H^{(\ell)}u_{k-\ell}-\mathcal P^\bot\mathcal O_1v^\bot_{k-1}\right),
\]
and $v_k\in N$ will be determined later. Observe that $v_k^\bot$ has the form
\[
v_k^\bot=-\mathcal L_{x_0}^{-1}\mathcal P^\bot\mathcal O_1v_{k-1}+f_k,
\]
where $f_k$ is known:
\[
f_k=\mathcal L_{x_0}^{-1}\left(-\mathcal P^\bot\mathcal H^{(k)}u_0 +\sum_{\ell=2}^{k-1}\lambda_\ell v^\bot_{k-\ell}-\sum_{\ell=2}^{k-1} \mathcal P^\bot \mathcal H^{(\ell)}u_{k-\ell}\right).
\]
This completes the iteration step. 
\medskip
\par
Thus, we have constructed an approximate eigenfunction $u_t$ of the operator $\mathcal H_t$ admitting an asymptotic expansion in the form of a formal asymptotic series in powers of $t$
\[
u_t=\sum_{\ell=0}^\infty u_\ell t^{\ell}, \quad u_\ell\in
{\cS(\RR^{2n})},
\]
with the corresponding formal eigenvalue
\[
\lambda_t=\sum_{\ell=0}^\infty\lambda_\ell t^{\ell},
\]
For any $N\in \NN$, consider
\[
u_t^{(N)}=\sum_{\ell=0}^N u_\ell t^{\ell}, \quad
\lambda_t^{(N)}=\sum_{\ell=0}^N\lambda_\ell t^{\ell}.
\]
Then we have
\[
\mathcal H_tu_t^{(N)}-\lambda_t^{(N)}u_t^{(N)}=\mathcal O(t^{N+1}).
\]

The constructed functions $u_t^{(N)}$ have sufficient decay properties. Therefore, by using \eqref{scaling}, we obtain the desired functions $\phi^p_{jN}\in C^\infty(X)$, which satisfy \eqref{e:Hh2} or \eqref{e:Hh}, supported in $B^X(x_0,a^X)\cong B^{T_{x_0}X}(0,a^X)$ and given by 
\[
\phi^p_{jN}(Z)=\chi(Z)\kappa^{-1/2}(Z)u_{1/\sqrt{p}}^{(N)}(\sqrt{p}Z), 
\]
where $\chi\in C^\infty_c(B^{T_{x_0}X}(0,a^X))$ is a fixed cut-off function.

\section{Lower bounds}\label{s:lower}   
In this section, we prove lower bounds for the eigenvalues of the Bochner Laplaican $\Delta^{L^p}$ and complete the proof of Theorem~\ref{t:Bochner-expansion}.  We start in Section~\ref{s:agmon} with Agmon type estimates and localization results for the corresponding eigenfuctions. In Section \ref{s:reduction} we apply these results to show that, under weak assumptions, the asymptotic behavior of the low-lying eigenvalues of $\Delta^{L^p}$ is equivalent to the asymptotic behavior of the low-lying eigenvalues of $P_{\mathcal H_p}\Delta^{L^p}P_{\mathcal H_p}$, where $P_{\mathcal H_p}$ is the generalized Bergman projection of the renormalized Bochner Laplacion.  Finally, in Section \ref{s:3.3} we use the fact that the operator $P_{\mathcal H_p}\Delta^{L^p}P_{\mathcal H_p}$ belongs to the algebra of Toeplitz operators on $X$ associated with the renormalized Bochner Laplacian $\Delta_p$, which was constructed in \cite{ioos-lu-ma-ma,bergman}. The asymptotic behavior, in the semiclassical limit, of low-lying eigenvalues of a self-adjoint Toeplitz operator under assumption that its principal symbol has a non-degenerate minimum with discrete wells was studied in \cite{toeplitz}. We apply these results to get lower estimates for the eigenvalues of $P_{\mathcal H_p}\Delta_pP_{\mathcal H_p}$ and complete the proof of Theorem~\ref{t:Bochner-expansion}.

\subsection{Agmon estimates}\label{s:agmon}
First, we recall a general lower bound due to Ma and Marinescu \cite{ma-ma02}. There exist $p_0\in \NN$ and $\alpha_0\in C^\infty(X)$ such that for any $p\geq p_0$ 
\begin{equation}\label{e:mama02}
(\Delta^{L^p}u,u)\geq \int_X(p\tau(x)+\alpha_0(x)) |u(x)|^2dv_X(x), \quad u\in C^\infty(X,L^p).
\end{equation}
It follows from Weitzenb\"ock formula. 

Let $\lambda_p$ be a sequence of eigenvalues of $\Delta^{L^p}$ such that 
\begin{equation}\label{e:mupO1}
\lambda_p=p\tau_0+\mu_p,\quad \mu_p=\mathcal O(1)
\end{equation}
and $u_p$ be an associated normalized eigenfunction. 

Denote by $\phi_p(x)$ the Agmon distance of $x\in X$ to the well 
\[
U_p=\{y\in X : p(\tau(y)-\tau_0)-\mu_p +\alpha_0(y)\leq 0\} 
\] 
associated to the Agmon metric 
\[
G_p=[p(\tau(y)-\tau_0)-\mu_p+\alpha_0(y)]_+g.
\] 
Recall the standard notation $[x]_+=\max(x,0)$, $[x]_-=\max(-x,0)$ for $x\in \RR$.

\begin{prop}
For any $\varepsilon \in (0,1)$, we have
\begin{multline}\label{e:Agmon}
(1-\varepsilon^2) \int_X  [p(\tau(x)-\tau_0)-\mu_p+\alpha_0(x)]_+ e^{2\varepsilon\phi_p(x)} |u_p(x)|^2 dv_X(x) \\ \leq \int_X [p(\tau(x)-\tau_0)-\mu_p+\alpha_0(x)]_- |u_p(x)|^2 dv_X(x). 
\end{multline}
\end{prop}

\begin{proof}
As in \cite{HM96}, we have the formula
\begin{multline*}
(\Delta^{L^p}(e^{\varepsilon \phi_p} u_p), e^{\varepsilon \phi_p} u_p)_{L^2(X,L^p)}
%=(e^\phi A_pu, e^\phi u) + ([e^\phi, \Delta^{L^p}_{D}]u,u). 
\\ =\lambda_p \|e^{\varepsilon\phi_p} u_p\|^2_{L^2(X,L^p)} + \varepsilon^2
\|e^{\varepsilon \phi_p} u_p d\phi_p\|^2_{L^2(X,L^p\otimes T^*X)}.
\end{multline*}
We observe that 
\[
|d\phi_p(x)|_{g^{T^*X}}\leq [p(\tau(x)-\tau_0)-\mu_p+\alpha_0(x)]_+, \quad x\in X. 
\]
By \eqref{e:mama02}, we have
\[
(\Delta^{L^p}(e^{\varepsilon\phi_p} u_p), e^{\varepsilon\phi_p} u_p)_{L^2(X,L^p)}\geq ((p\tau+\alpha_0)e^{\varepsilon\phi_p} u_p, e^{\varepsilon\phi_p} u_p)_{L^2(X,L^p)}.
\]
Using these facts, one can easily complete the proof. 
\end{proof}

Now we replace Assumption \ref{a:2} with a more general assumption of non-degenerate submanifold wells. 

\begin{ass}\label{a:3}
The minimum set 
\[
W=\{x_0\in X : \tau(x_0)=\tau_0\}
\]  
is a smooth submanifold of $X$ and there exists $a_0>0$ such that
\begin{equation}\label{e:tauW}
\tau(x)-\tau_0\geq a_0d(x,W)^2,\quad x\in X, 
\end{equation}
where $d(x,W)$ is the geodesic distance from $x$ to $W$.  
\end{ass}

For any $\varepsilon>0$, denote by $W_\varepsilon$ the $\varepsilon$-neighborhood of $W$:
\[
W_\varepsilon=\{x\in X : d(x,W)\leq \varepsilon\}.
\]

\begin{lem}\label{l:exp}
Suppose that Assumption \ref{a:3} holds. For any $\delta>0$, there exist $p_0\in \NN$, $c>0$, $C>0$ and $\alpha>0$ such that, for any $p>p_0$
\[
\int_{X\setminus W_{cp^{-1/4+\delta}}} |u_p(x)|^2 dv_X(x)\leq Ce^{-\alpha p^{2\delta}}. 
\]
\end{lem}

\begin{proof}
Observe that there exist $C_1>0$ and $C_2>0$ such that if $d(x,W)>C_1p^{-1/2}$, then 
\begin{equation}\label{e:phi}
\phi_p(x)\geq C_2 p^{1/2}d(x,W)^2. 
\end{equation}
Indeed, by \eqref{e:tauW} and \eqref{e:mupO1}, there exists $a>0$ such that 
\[
p(\tau(x)-\tau_0)-\mu_p+\alpha_0(x)\geq a_0pd(x,W)^2-a, \quad x\in X
\]
Therefore, if $d(x,W)>C_1p^{-1/2}$ with $C_1=(2a/a_0)^{1/2}$, then 
\[
p(\tau(x)-\tau_0)-\mu_p+\alpha_0(x)\geq \frac 12a_0pd(x,W)^2.
\]
In particular, for such $x$, 
\begin{equation}\label{e:upper}
p(\tau(x)-\tau_0)-\mu_p+\alpha_0(x)\geq a.
\end{equation}
It remains to observe that, for the metric $G_W=d(x,W)^2g$, the distance function $d_W$ to $W$ satisfies the estimate
\[
d_W(x)>C_0d(x,W)^2.
\]
with some $C_0>0$. It completes the proof of \eqref{e:phi}.

Using \eqref{e:tauW}, it is easy to see that there exists $c_1>0$ such that, if $p(\tau(x)-\tau_0)-\mu_p+\alpha_0(x)<0$, then $d(x,W)<c_1p^{-1/2}$. In other words, we have an inclusion 
\begin{equation}\label{e:UpW}
U_p\subset W_{c_1p^{-1/2}}.
\end{equation} 
Without loss of generality, we can assume that $c_1>C_1$, that is, \eqref{e:phi} holds for any $x\in X\setminus W_{c_1p^{-1/2}}$. 

It is clear that there exists $a_1>0$ such that
\[
\tau(x)-\tau_0 \leq a_1d(x,W)^2,\quad x\in X.
\]
Therefore. there exists $K>0$ such that, for any $p$ and  $x\in U_p$,
\[
|p(\tau(x)-\tau_0)-\mu_p+\alpha_0(x)|<K.
\]
Using this estimate, for the right hand side of \eqref{e:Agmon}, we get
\begin{multline}\label{e:right}
\int_X [p(\tau(x)-\tau_0)-\mu_p+\alpha_0(x)]_- |u_p(x)|^2 dv_X(x) \\ \leq K \int_{U_p} |u_p(x)|^2 dv_X(x)\leq K. 
\end{multline}
For the left hand side of \eqref{e:Agmon}, using \eqref{e:phi}, \eqref{e:upper} and \eqref{e:UpW} and setting $\varepsilon=1/2$, we get
\begin{multline}\label{e:left}
\int_X  [p(\tau(x)-\tau_0)-\mu_p+\alpha_0(x)]_+ e^{2\varepsilon\phi_p(x)} |u_p(x)|^2 dv_X(x) \\ \geq a\int_{X\setminus W_{c_1p^{-1/2}}} e^{C_2 p^{1/2} d(x,W)^2} |u_p(x)|^2 dv_X(x). 
\end{multline}
Combining \eqref{e:right} and \eqref{e:left}, we infer that 
\begin{equation}\label{e:Ka}
\int_{X\setminus W_{c_1p^{-1/2}}} e^{C_2 p^{1/2}d(x,W)^2} |u_p(x)|^2 dv_X(x)\leq \frac{K}{a}.
\end{equation}
Take any $\delta>0$. Then, if $d(x,W)>cp^{-1/4+\delta}$, then $p^{1/2} d(x,W)^2>cp^{2\delta}$ and there exists $p_0\in \NN$ such that $d(x,W)>c_1p^{-1/2}$ for any $p>p_0$. Therefore, for any $p>p_0$, we get
\begin{multline*}
\int_{X\setminus W_{cp^{-1/4+\delta}}} e^{C_2cp^{2\delta}} |u_p(x)|^2 dv_X(x)\\ <\int_{X\setminus W_{c_1p^{-1/2}}} e^{C_2 p^{1/2} d(x,W)^2} |u_p(x)|^2 dv_X(x).
\end{multline*}
Combining with \eqref{e:Ka}, this immediately completes the proof.
\end{proof}

\subsection{Reduction to the lowest Landau level}\label{s:reduction}
The renormalized Bochner Laplacian $\Delta_p$ is a second order differential opertor acting on $C^\infty(X,L^p)$ by
 \begin{equation}\label{e:defDp}
\Delta_p=\Delta^{L^p}-p\tau.
\end{equation}
This operator was introduced by V. Guillemin and A. Uribe in \cite{Gu-Uribe}. When $(X,\omega)$ is a Kaehler manifold, it is twice the corresponding Kodaira Laplacian on functions $\Box^{L^p}=\bar\partial^{L^p*}\bar\partial^{L^p}$.
 
Denote by $\sigma(\Delta_p)$ the spectrum of $\Delta_p$ in $L^2(X,L^p)$. Put
  \begin{equation}\label{e:def-mu0}
 \mu_0=\inf_{u\in T_xX, x\in X}\frac{iR^L_x(u,J(x)u)}{|u|_g^2}.
 \end{equation}
By \cite[Cor. 1.2]{ma-ma08}, there exists a constant $C_L>0$ such that for any $p$
   \begin{equation}\label{e:gap}
 \sigma(\Delta_p)\subset [-C_L,C_L]\cap [2p\mu_0-C_L,+\infty).
  \end{equation}
Consider the finite-dimensional vector subspace $\mathcal H_p\subset L^2(X,L^p)$ spanned by the eigensections of $\Delta_p$ corresponding to eigenvalues in $[-C_L,C_L]$.  Let $P_{\mathcal H_p}$ be the orthogonal projection from $L^2(X,L^p)$ onto $\mathcal H_p$ (the generalized Bergman projection of $\Delta_p$). 

\begin{prop}\label{p::lower-lambda} 
Suppose that Assumption \ref{a:3} holds. Let $\lambda_p$ be a sequence of eigenvalues of $\Delta^{L^p}$ such that 
\[
\lambda_p=p\tau_0+\mu_p,\quad \mu_p=\mathcal O(1)
\] 
and let $u_p$ be an associated normalized eigenfunction. Then, for any $\delta\in (0,1/2)$, 
\begin{equation}\label{e:lower-lambda}
(P_{\cH_p}\Delta^{L^p} P_{\cH_p}u_p,u_p)\leq \lambda_p+\mathcal O(p^{-1/2+\delta}).
\end{equation}
\end{prop}

\begin{proof}
By the IMS localization formula, we have
\begin{multline}\label{e:IMS}
\lambda_p=(\Delta^{L^p}u_p,u_p)=(P_{\cH_p}\Delta^{L^p} P_{\cH_p}u_p,u_p)\\ +((I-P_{\cH_p})\Delta^{L^p} (I-P_{\cH_p})u_p,u_p) +p([\tau, [\tau, P_{\cH_p}]]u_p,u_p).
\end{multline}

Let us estimate the last term in the right-hand side of \eqref{e:IMS}.

\begin{lem}\label{l:ddcom}
For any $\delta\in (0,1/2)$, there exists $C_1>0$ and $p_1\in \NN$ such that, for all $p>p_1$,
\[
|([\tau, [\tau, P_{\cH_p}]]u_p,u_p)|\leq C_1p^{-3/2+2\delta}.   
\]
\end{lem}
 
\begin{proof}
It is easy to see that
\[
([\tau, [\tau, P_{\cH_p}]]u_p, u_p) =\int_X\int_X (\tau(x)-\tau(y))^2P_p(x,y)u_p(x)\overline{u_p(y)}dv_X(x)dv_X(y),
\]
where $P_p$ is the smooth kernel of the operator $P_{\cH_p}$ (the generalized Bergman kernel of $\Delta_p$). 

By Lemma \ref{l:exp}, we have, for any $\delta>0$, 
\begin{multline*}
([\tau, [\tau, P_{\cH_p}]]u_p, u_p) \\ =\int_{W_{cp^{-1/4+\delta}}}\int_{W_{cp^{-1/4+\delta}}} (\tau(x)-\tau(y))^2P_p(x,y)u_p(x)\overline{u_p(y)}dv_X(x)dv_X(y)\\ 
+\mathcal O(e^{-\alpha p^{2\delta}}).
\end{multline*}
By Schur's test, we conclude
\begin{multline*}
([\tau, [\tau, P_{\cH_p}]]u_p, u_p) \\ \leq \sup_{x\in W_{cp^{-1/4+\delta}}}\int_{W_{cp^{-1/4+\delta}}} (\tau(x)-\tau(y))^2P_p(x,y)dv_X(y)
+\mathcal O(e^{-\alpha p^{2\delta}}),
\end{multline*}
so it remains to estimate the integral 
\[
I=\int_{W_{cp^{-1/4+\delta}}} (\tau(x)-\tau(y))^2P_p(x,y)dv_X(y)
\]
uniformly on $x\in W_{cp^{-1/4+\delta}}$.

By \cite{Ko-ma-ma}, there exist $C>0$ and $\beta>0$ such that, for any $p\geq 1$, we have 
\[
|P_{p}(x,y)|\leq Cp^{n}\exp(-\beta\sqrt{p}d(x,y)), \quad x,y\in X. 
\]
Since $d\tau_{x_0}=0$ for any $x_0\in W$, for any $x\in W_{cp^{-1/4+\delta}}$ and $y\in W_{cp^{-1/4+\delta}}$, we have an estimate
\[
|\tau(x)-\tau(y)|\leq \sup_{z\in W_{cp^{-1/4+\delta}}} |d\tau_z|\, d(x,y) \leq  cp^{-1/4+\delta}d(x,y). 
\]
Thus, we get 
\begin{multline*}
I\leq C_2p^{-1/2+2\delta}\sup_{x\in W_{cp^{-1/4+\delta}}}\int_X p^n d(x,y)^2 \exp(-\beta\sqrt{p}d(x,y))\,dv_X(y)\\ \leq C_3p^{-3/2+2\delta},
\end{multline*}
that completes the proof of the lemma. 
 \end{proof}

Fix $\delta\in (0,1/2)$. By \eqref{e:IMS} and Lemma~\ref{l:ddcom}, we have 
\[
(\Delta^{L^p} P_{\cH_p}u_p,P_{\cH_p}u_p)+ (\Delta^{L^p}(I-P_{\cH_p})u_p, (I-P_{\cH_p})u_p)\leq \lambda_p+\mathcal O(p^{-1/2+\delta})
\]
By \eqref{e:defDp}, for the left-hand side of the last estimate, we have
\begin{multline*}
(\Delta^{L^p} P_{\cH_p}u_p,P_{\cH_p}u_p)+ (\Delta^{L^p}(I-P_{\cH_p})u_p, (I-P_{\cH_p})u_p) \\ 
\begin{aligned} 
= &(\Delta_p P_{\cH_p}u_p,P_{\cH_p}u_p)+ (\Delta_p(I-P_{\cH_p})u_p, (I-P_{\cH_p})u_p)\\ & +p(\tau P_{\cH_p}u_p,P_{\cH_p}u_p)+ p(\tau (I-P_{\cH_p})u_p, (I-P_{\cH_p})u_p)\\ \geq &(\Delta_p P_{\cH_p}u_p,P_{\cH_p}u_p)+ (\Delta_p(I-P_{\cH_p})u_p, (I-P_{\cH_p})u_p) +p\tau_0.
\end{aligned}
\end{multline*}
Therefore, it gives 
\begin{multline}\label{e:D1}
(\Delta_p P_{\cH_p}u_p,P_{\cH_p}u_p)+ (\Delta_p(I-P_{\cH_p})u_p, (I-P_{\cH_p})u_p) \\ \leq \mu_p+\mathcal O(p^{-1/2+\delta})=\mathcal O(1).
\end{multline}
By \eqref{e:gap} and the definition of $P_{\cH_p}$, we get estimates
\begin{equation}\label{e:D2a}
(\Delta_pP_{\cH_p}u_p, P_{\cH_p}u_p) \geq -C_L\|P_{\cH_p}u_p\|^2\geq -C_L
\end{equation}
and  
\begin{equation}\label{e:D2}
(\Delta_p(I-P_{\cH_p})u_p, (I-P_{\cH_p})u_p) \geq (2p\mu_0-C_L)\|(I-P_{\cH_p})u_p\|^2.
\end{equation}
By \eqref{e:D1} and \eqref{e:D2a}, it follows that
\begin{equation}\label{e:D1a}
(\Delta_p(I-P_{\cH_p})u_p, (I-P_{\cH_p})u_p) \leq \mu_p+\mathcal O(p^{-1/2+\delta})=\mathcal O(1)
\end{equation}
From \eqref{e:D1a} and \eqref{e:D2}, we infer that
\[
\|(I-P_{\cH_p})u_p\|=\mathcal O(p^{-1/2})\ \text{and}\ 
u_p=P_{\cH_p}u_p+\mathcal O(p^{-1/2}). 
\]
Using these facts, by \eqref{e:IMS} and Lemma~\ref{l:ddcom}, we get \eqref{e:lower-lambda}.
\end{proof}

\begin{prop}\label{t:Bochner-estimates2}
Under Assumption \ref{a:3}, if for any $j\in \NN$ we have
\begin{equation}\label{e:lj-rough}
\lambda_j(\Delta^{L^p})=p\tau_0+\mathcal O(1),
\end{equation}
then, for any $j\in \NN$ and $\delta\in (0,1/2)$, there exists $C>0$ such that 
\[
\lambda_j(P_{\mathcal H_p}\Delta^{L^p}P_{\mathcal H_p})-Cp^{-1/2+\delta} \leq \lambda_j(\Delta^{L^p})\leq \lambda_j(P_{\mathcal H_p}\Delta^{L^p}P_{\mathcal H_p}), \quad p\to\infty.
\]
\end{prop}

\begin{proof}
The first inequality is an immediate consequence of Proposition~\ref{p::lower-lambda} applied to the sequence $\lambda_p=\lambda_j(\Delta^{L^p})$ and the mini-max principle. The second inequality is well-known (see, for instance, \cite[Theorem XIII.3]{RSIV}), the Rayleigh-Ritz technique).  
\end{proof}

The fact that the eigenvalues $\lambda_j(\Delta^{L^p})$ satisfy the condition \eqref{e:lj-rough} under Assumption~\ref{a:2} is an easy consequence of Theorem~\ref{t:qmodes} (see Remark~\ref{r:upper-bound}). Actually, one can show that they satisfy the condition \eqref{e:lj-rough} under Assumption~\ref{a:3}. The proof will be given in a forthcoming paper. 

\subsection{The case of discrete wells}\label{s:3.3}
Now we return to the case of discrete wells and complete the proof of Theorem~\ref{t:Bochner-expansion}. So we suppose that Assumption~\ref{a:2} holds. Proposition~\ref{t:Bochner-estimates2} reduces our considerations to the operator $P_{\mathcal H_p}\Delta^{L^p}P_{\mathcal H_p}$. 

Now we will use the algebra of Toeplitz operators on $X$ associated with the renormalized Bochner Laplacian $\Delta_p$, which was constructed in \cite{ioos-lu-ma-ma,bergman}.
By definition, a Toeplitz operator is a sequence $\{T_p\}=\{T_p\}_{p\in \mathbb N}$ of bounded linear operators $T_p : L^2(X,L^p)\to L^2(X,L^p)$, such that $T_p=P_{\mathcal H_p}T_pP_{\mathcal H_p}$ for any $p\in \mathbb N$ and there exists a sequence $g_l\in C^\infty(X)$ such that 
\[
T_p=P_{\mathcal H_p}\left(\sum_{l=0}^\infty p^{-l}g_l\right)P_{\mathcal H_p}+\mathcal O(p^{-\infty}),
\]
which means that, for any natural $k$, there exists $C_k>0$ such that, for any $p\in \mathbb N$, 
\[
\left\|T_p-P_{\mathcal H_p}\left(\sum_{l=0}^k p^{-l}g_l\right)P_{\mathcal H_p}\right\|\leq C_kp^{-k-1}.
\]

One can write
\begin{equation}\label{e:PDP}
p^{-1}P_{\mathcal H_p}\Delta^{L^p}P_{\mathcal H_p}=p^{-1} P_{\mathcal H_p}\Delta_pP_{\mathcal H_p}+P_{\mathcal H_p}(\tau-\tau_0) P_{\mathcal H_p}+\tau_0 P_{\mathcal H_p}.
\end{equation}
The crucial fact is that the operator $P_{\mathcal H_p}\Delta_pP_{\mathcal H_p}=\Delta_pP_{\mathcal H_p}$ is a Toeplitz operator. This was proved in \cite{bergman}, extending previous results of \cite{ma-ma08}. Therefore, the operator $p^{-1}P_{\mathcal H_p}\Delta^{L^p}P_{\mathcal H_p}$ is a Toeplitz operator in the above sense.

In \cite{toeplitz}, the author studied the asymptotic behavior, in the semiclassical limit, of low-lying eigenvalues of a self-adjoint Toeplitz operator under assumption that its principal symbol has a non-degenerate minimum with discrete wells. We can apply these results to get lower estimates for $\lambda_j(P_{\mathcal H_p}\Delta_pP_{\mathcal H_p})$. Let us briefly recall 
the relevant results from \cite{toeplitz}.

Suppose that $T_p$ is a self-adjoint Toeplitz operator:
\begin{equation}\label{e:Tii}
T_p=P_{\mathcal H_p}\left(\sum_{l=0}^\infty p^{-l}g_l\right)P_{\mathcal H_p}+\mathcal O(p^{-\infty}),  
\end{equation}
with principal symbol $g_0=h$. Assume that the principal symbol $h$ satisfies the condition:
\begin{equation}\label{e:minh=0}
\min_{x\in X}h(x)=0.
\end{equation}

For each non-degenerate minimum $x_0$ of $h$:
\[
h(x_0)=0, \quad {\rm Hess}\,h(x_0)>0,
\]
one can define the model operator for $\{T_{p}\}$ at $x_0$ in the following way. 

Recall that $\mathcal L_{x_0}$ is a second order differential operator  in $C^\infty(T_{x_0}X)$ given by \eqref{e:defL0} and $\mathcal P_{x_0}$ is the orthogonal projection in $L^2(T_{x_0}X)$ to the kernel of $\mathcal L_{x_0}$. The model operator for $\{T_{p}\}$ at $x_0$ is the Toeplitz operator $\mathcal T_{x_0}$ in $L^2(T_{x_0}X)$ given by
\begin{equation}\label{e:model-T}
\mathcal T_{x_0}=\mathcal P_{x_0}(q_{x_0}(Z)+g_1(x_0))\mathcal P_{x_0},
\end{equation}
where  
\[
q_{x_0}(Z)=\left(\frac12 {\rm Hess}\,h(x_0)Z,Z\right)
\]
is a positive quadratic form on $T_{x_0}X\cong \mathbb R^{2n}$ and $g_1$ is the second coefficient in the asymptotic expansion \eqref{e:Tii}. 

Now suppose that a self-adjoint Toeplitz operator $T_p$ with the principal symbol $h$, satisfying \eqref{e:minh=0}, such that each minimum is non-degenerate. Then the set $U_0=h^{-1}(0)$ is a finite set of points:
\[
U_0=\{x_1,\ldots, x_N\}.
\]
Let $\mathcal T$ be the self-adjoint operator on $L^2(T_{x_1}X)\oplus \ldots \oplus  L^2(T_{x_N}X)$ defined by 
\[
\mathcal T=\mathcal T_{x_1}\oplus \ldots \oplus \mathcal T_{x_N}.
\]

Let $\{\lambda^m_p\}$ be the increasing sequence of the eigenvalues of $T_{p}$ on $\mathcal H_p$ (counted with multiplicities) and $\{\mu_m\}$  
 the increasing sequence of the eigenvalues of $\mathcal T$ (counted with multiplicities). 
By Theorem 1.5 of \cite{toeplitz}, for any fixed $m$, $\lambda^m_p$ has an asymptotic expansion, when $p\to \infty$, of the form
\begin{equation}\label{e:DelB}
\lambda^m_p=p^{-1}\mu_m+p^{-3/2}\phi_m+\mathcal O(p^{-2})
\end{equation}
with some $\phi_m\in \mathbb R$. 

Now we apply this result to the Toeplitz operator 
\[
D_p=p^{-1} P_{\mathcal H_p}\Delta_pP_{\mathcal H_p}+P_{\mathcal H_p}(\tau-\tau_0) P_{\mathcal H_p}
\]
under Assumptions 1 and 2. Its principal symbol $h(x)=\tau(x)-\tau_0$ satisfies \eqref{e:minh=0} and each minimum is non-degenerate, $U_0=W$. One can easily find the first two terms of the asymptotic expansion \eqref{e:Tii} for the operator $D_p$:
\[
D_p=P_{\mathcal H_p}\left(\tau(x)-\tau_0+p^{-1}d_1\right)P_{\mathcal H_p}+\mathcal O(p^{-2}),
\]
where $d_1$ is the principal symbol of the operator $P_{\mathcal H_p}\Delta_pP_{\mathcal H_p}$. 

To find $d_1$, we use a near-diagonal asymptotic expansion of the smooth kernel $P_{1,p}(x,x^\prime)$, $x,x^\prime\in X$ of the operator $\Delta_pP_{\mathcal H_p}$ with respect to the Riemannian volume form $dv_X$ \cite{ma-ma08,bergman}. In normal coordinates near an arbitrary point $x_0\in X$ introduced in Section \ref{s:rescaling}, the kernel $P_{1,p}(x,x^\prime)$ induces a smooth function $P_{1,p,x_0}(Z,Z^\prime)$ on the set of all $Z,Z^\prime\in T_{x_0}X$ with $x_0\in X$ and $|Z|, |Z^\prime|<a_X$.  
By \cite[Theorem 1.1]{bergman}, for any $k\in \mathbb N$ and $x_0\in X$,   we have  
\begin{multline}\label{e:Pqr}
p^{-n}P_{1,p,x_0}(Z,Z^\prime)\cong 
\sum_{r=2}^kF_{1,r,x_0}(\sqrt{p} Z, \sqrt{p}Z^\prime)\kappa^{-\frac 12}(Z)\kappa^{-\frac 12}(Z^\prime)p^{-\frac{r}{2}+1},
\end{multline}
This means that there exist $\varepsilon\in (0,a_X]$ and $C_0>0$ with the following property:
for any $l\in \mathbb N$, there exist $C>0$ and $M>0$ such that for any $x_0\in X$, $p\geq 1$ and $Z,Z^\prime\in T_{x_0}X$, $|Z|, |Z^\prime|<\varepsilon$, we have 
\begin{multline*}
\Bigg|p^{-n}P_{1,p,x_0}(Z,Z^\prime)\kappa_{x_0}^{\frac 12}(Z)\kappa_{x_0}^{\frac 12}(Z^\prime) -\sum_{r=2}^kF_{1,r,x_0}(\sqrt{p} Z, \sqrt{p}Z^\prime)p^{-\frac{r}{2}+1}\Bigg|_{\mathcal C^{l}(X)}\\ 
\leq Cp^{-\frac{k+1}{2}}(1+\sqrt{p}|Z|+\sqrt{p}|Z^\prime|)^M\exp(-\sqrt{C_0p}|Z-Z^\prime|)+\mathcal O(p^{-\infty}).
\end{multline*}

By \cite[Theorem 1.18]{ma-ma08}, the function $F_{1,2,x_0}(Z,Z^\prime)$ coincides with the Schwartz kernel of the operator $F_{1,2}$ defined by \eqref{e:F12}. Moreover, it has the form 
\begin{equation}\label{e:Fq2}
F_{1,2,x_0}(Z,Z^\prime)=J_{1,2,x_0}(Z,Z^\prime)\mathcal P_{x_0}(Z,Z^\prime),
\end{equation}
where $J_{1,2,x_0}(Z,Z^\prime)$ is an even polynomial in $Z, Z^\prime$. Finally, as shown in \cite[Theorem 6.5]{bergman} and \cite[Proposition 4.3]{ioos-lu-ma-ma}, for all $x_0\in X$ and $Z,Z^\prime\in T_{x_0}X$,
\begin{equation}\label{e:J12}
J_{1,2,x_0}(Z,Z^\prime)=J_{1,2,x_0}(0,0)=:J_{1,2}(x_0),
\end{equation}
and the principal symbol $d_1$ of $\Delta_pP_{\mathcal H_p}$ is given by 
\begin{equation}\label{e:prinTp}
d_1=J_{1,2}.
\end{equation}

Thus, we see that the model operator for the Toeplitz operator $D_{p}$ at $x_0$ given by \eqref{e:model-T} coincides with the operator $\mathcal D_{x_0}$ given by \eqref{e:model-D} with $J_{1,2}$ defined by \eqref{e:J12}.

Applying Theorem 1.5 of \cite{toeplitz} (cf. \eqref{e:DelB}) to the operator $D_p$, we infer that
\[
\lambda_j(P_{\cH_p}(\Delta_p+p(\tau-\tau_0))P_{\cH_p})=\mu_j+\mathcal O(p^{-1/2}),
\]
where $\mu_j$ are the eigenvalues of the operator $\cD$ given by \eqref{e:def-Tau}.  By \eqref{e:PDP}, it follows that
\[
\lambda_j(P_{\mathcal H_p}\Delta^{L^p}P_{\mathcal H_p}) =p\tau_0+\mu_j+\mathcal O(p^{-1/2}).
\]
By Proposition \ref{t:Bochner-estimates2},  for any $j\in \NN$ and $\delta\in (0,1/2)$, there exists $C>0$ such that 
\begin{equation}\label{e:lowerlj}
\lambda_j(\Delta^{L^p})\geq p\tau_0+\mu_j-Cp^{-1/2+\delta}.
\end{equation}
Theorem~\ref{t:Bochner-expansion} follows from this estimate and Theorem \ref{t:qmodes} with a standard argument. Let us briefly recall it. 

Fix $j\in \NN$ and $\delta\in (0,1/2)$. By \eqref{e:lowerlj} and Theorem~\ref{t:qmodes}, there exist $C_{j1}>0$ and $p_{j1}\in \NN$ such that, for any $p>p_{j1}$,
\[
I_j\; \bigcap \; \sigma(\Delta^{L^p}) =\{\lambda_j(\Delta^{L^p})\},
\]
where
\[
I_j=\left(p\tau_0+\mu_j-C_{j1}p^{-1/2+\delta},  p\tau_0+\mu_j+C_{j1}p^{-1/2}\right).
\]
On the other hand, by Corollary~\ref{c:dist}, for any natural $j$, there exist $C_{j2}>0$ and
$p_{j2}>0$ such that, for any $p>p_{j2}$,
\[
{\rm dist}(p\tau_0+\mu_j, \sigma(\Delta^{L^p}))\leq C_{j2}p^{-1/2}.
\]
Without loss of generality, we can assume that, for any $p>\max(p_{j1}, p_{j2})$,
\[
(p\tau_0+\mu_j-C_{j2}p^{-1/2}, p\tau_0+\mu_j+C_{j2}p^{-1/2})\cap I_\ell=\emptyset, \forall \ell\neq j.
\]
Hence, for any $p>\max(p_{j1}, p_{j2})$, $\lambda_j(\Delta^{L^p})$ is the point of ${\rm Spec}(\Delta^{L^p})$, closest to $p\tau_0+\mu_j$. It follows that
\[
|\lambda_j(\Delta^{L^p})-(p\tau_0+\mu_j)|\leq C_{j2}p^{-1/2}, \quad
p>\max(p_{j1}, p_{j2})\,,
\]
that proves \eqref{e:l1}. 

The expansion \eqref{e:l1s} is proved similarly. 

In the end of this section, we recall the formula for $F_{1,2}$ obtained in \cite[Subsection 2.1]{ma-ma08}. Put
\[
\frac{\partial}{\partial z_j}=\frac{1}{2}\left(\frac{\partial}{\partial Z_{2j-1}}-i\frac{\partial}{\partial Z_{2j}}\right), \quad \frac{\partial}{\partial \overline{z}_j}=\frac{1}{2}\left(\frac{\partial}{\partial Z_{2j-1}}+i\frac{\partial}{\partial Z_{2j}}\right).
\]
Define first order differential operators $b_j,b^{+}_j, j=1,\ldots,n,$ on $T_{x_0}X$ by
\[
b_j= -2\nabla_{\tfrac{\partial}{\partial z_j}}-R^L_{x_0}(\mathcal R, \tfrac{\partial}{\partial z_j})=-2{\tfrac{\partial}{\partial z
_j}}+\frac{1}{2}a_j\overline{z}_j,
\]
\[
b^{+}_j= 2\nabla_{\tfrac{\partial}{\partial \overline{z}_j}} + R^L_{x_0}(\mathcal R, \tfrac{\partial}{\partial \overline{z}_j})=2{\tfrac{\partial}{\partial\overline{z}_j}}+\frac{1}{2}a_j z_j.
\]
So we can write
\[
\cL_{x_0}=\sum_{j=1}^n b_j b^{+}_j,\quad \tau(x_0)= \sum_{j=1}^n a_j. 
\]
Then we have (see (2.12) in \cite{ma-ma08})
\[
F_{1,2,x_0}(Z,Z^\prime) = [\cP_{x_0}\mathcal F_{1,2,x_0}\cP_{x_0}] (Z,Z^\prime),
\]
where $\mathcal F_{1,2,x_0}$ is an unbounded linear operator in $L^2(T_{x_0}X)$ given by 
\begin{multline*} 
\mathcal F_{1,2,x_0}
= 4 \left \langle R^{TX}_{x_0} \left(\frac{\partial}{\partial z_j}, \frac{\partial}{\partial z_k}\right) \frac{\partial}{\partial \overline{z}_j}, 
\frac{\partial}{\partial \overline{z}_k}\right \rangle\\
+ \left \langle (\nabla ^{X} \nabla ^{X}\mathcal J)_{(\mathcal R,\mathcal R)} \frac{\partial}{\partial z_j},
 \frac{\partial}{\partial \overline{z}_j} \right \rangle 
+\frac{i}{4} 
\tr_{|TX} \Big(\nabla ^{X} \nabla ^{X}(J\mathcal J)\Big)_{(\mathcal R, \mathcal R)} \\
+ \frac{1}{9} |(\nabla_{\mathcal R}^X \mathcal J) \mathcal R |^2
+ \frac{4}{9} \sum_{j=1}^n\left\langle(\nabla ^{X}_{\mathcal R} \mathcal J) \mathcal R, \frac{\partial} {\partial z_j}\right\rangle b^+_j \cL_{x_0}^{-1}b_j 
\left\langle(\nabla ^{X}_{\mathcal R}\mathcal J)\mathcal R,\frac{\partial}
{\partial\overline{z}_j}\right\rangle.
\end{multline*}

A more explicit computation of $J_{1,2}$ will be given elsewhere.

\end{document}